\documentclass{birkjour}
 \newtheorem{thm}{Theorem}[section]
 \newtheorem{cor}[thm]{Corollary}
 \newtheorem{lem}[thm]{Lemma}
 
 \theoremstyle{definition}
 \newtheorem{defn}[thm]{Definition}
 \theoremstyle{remark}
 \newtheorem{rem}[thm]{Remark}
 
 \numberwithin{equation}{section}

\usepackage{hyperref}

\begin{document}

\newcommand{\id}{1\hspace{-0,9ex}1}
\newcommand{\sgn}{\text{sign}}
\newcommand{\wlim}{\text{w -}\lim}
%-------------------------------------------------------------------------
% editorial commands: to be inserted by the editorial office
%
%\firstpage{1} \volume{228} \Copyrightyear{2004} \DOI{003-0001}
%\seriesextraline{This is the Concrete Title of this Book\br H.E. R and S.T.C. W, Eds.}
%\firstpage{1}

\title[Asymptotic Results for a weighted $p$-Laplacian evolution Equation]
 {Asymptotic Results for Solutions of a weighted $p$-Laplacian evolution Equation with Neumann Boundary Conditions}

%----------Author 1
\author[Alexander Nerlich]{Alexander Nerlich}

\address{%
	Ulm University\\
	}

\email{alexander.nerlich@uni-ulm.de}

\subjclass{Primary 35B40; Secondary 47H20}

\keywords{weighted $p$-Laplacian evolution equation, asymptotic results for nonlinear semigroups, nonlocal diffusion, Neumann boundary conditions}

\date{June, 2017}

\begin{abstract}
The purpose of this paper is to investigate the time behavior of the solution of a weighted $p$-Laplacian evolution equation, given by 
\begin{align}
\label{eveq}
\begin{cases} u_{t} = \text{div} \left( \gamma |\nabla u|^{p-2}\nabla u \right) & \text{on } (0,\infty)\times  S , \\ 
\gamma|\nabla u|^{p-2}\nabla u\cdot\eta=0 & \text{on } (0,\infty)\times \partial S , \\
u(0,\cdot)=u_{0} & \text{on }   S ,\end{cases}
\end{align}
where $n \in \mathbb{N}\setminus \{1\}$, $p \in (1,\infty)\setminus \{2\}$, $S\subseteq \mathbb{R}^{n}$ is an open, bounded and connected set of class $C^{1}$, $\eta$ is the unit outer normal on $\partial  S $,  and $\gamma: S \rightarrow (0,\infty)$ is a bounded function which can be extended to an $A_{p}$-Muckenhoupt weight on $\mathbb{R}^{n}$.\\ 
It will be proven that the solution of (\ref{eveq}) converges in $L^{1}(S)$ to the average of the initial value $u_{0} \in L^{1}( S )$. Moreover, a conservation of mass principle, an extinction principle and a decay rate for the solution will be derived. 
\end{abstract}

%%% ----------------------------------------------------------------------
\maketitle
%%% ----------------------------------------------------------------------

\section{Introduction}
The initial value problem (\ref{eveq}) has been considered by F. Andreu, J.M. Maz\'{o}n,  J. Rossi and J. Toledo in \cite{main}, Section 3. More precisely it has been shown that this equation admits, for any integrable initial value, a unique entropy solution.\\ 
From the applied point of view, the solution $u$ can be used to model diffusion processes: One has some initially given quantity $u_{0}$ which changes over time due to an external force $\gamma$ and the resulting quantity at time $t$ is $u(t)$.\\ 
For example, as B. Birnir and J. Rowlett demonstrated in \cite{birnirtheory}, the solution $u$ of (\ref{eveq}) can be used to describe the evolution of a fluvial landscape $u_{0}$ (for example a hill) which changes over time due to rain that determines the water depth $\gamma$.\\ 

The basic technique used in \cite{main} to obtain the existence of a unique entropy solution of (\ref{eveq}) is to apply nonlinear semigroup theory. To be slightly more specific; the concept of entropy solution of (\ref{eveq})  is defined precisely in such a way that it coincides with the usual definition of strong solution of the evolution equation 
\begin{align}
\label{intro}
0 \in u^{\prime}(t)+\mathcal{A} u(t),~\text{a.e. } t \in (0,\infty) \text{ and } u(0)=u_{0},
\end{align}
where $\mathcal{A}:D( \mathcal{A} )\rightarrow 2^{L^{1}( S )}$ is a certain multi-valued operator to be specified later.\\

Once the existence of a unique strong solution of (\ref{intro}) has been recalled, which is subject of Section \ref{basicdef}, the results mentioned in the abstract will be proven.\\

The asymptotic results which will be proven, are formulated by means of nonlinear semigroup theory. Therefore, let $T(\cdot)u_{0}:[0,\infty)\rightarrow L^{1}( S )$ denote, for a given $u_{0}\in L^{1}( S )$, the uniquely determined strong solution of (\ref{intro}) corresponding to the initial value $u_{0}$. Moreover, let $\overline{(u_{0})}_{S}:=\frac{1}{\lambda( S )} \int \limits_{ S } u_{0} d\lambda$, where $\lambda$ denotes the Lebesgue measure.\\ 
Firstly, it will be proven that $T$ conserves mass, i.e. $\overline{(u_{0})}_{S}=\overline{(T(t)u_{0})}_{S}$, for all $t \in [0,\infty)$ and $u_{0} \in L^{1}( S ) $. In addition, one has
\begin{align}
\label{introbound_conv}
\lim \limits_{t \rightarrow \infty} ||T(t)u_{0}-\overline{(u_{0})}_{ S }||_{L^{q}( S )}=0,
\end{align}
for any  $u_{0}\in L^{q}( S )$ and  $q\in [1,\infty)$; as well as 
\begin{align}
\label{introbound_L1}
||T(t)u_{0}-\overline{(u_{0})}_{ S }||_{L^{1}( S )} \leq C ||u_{0}-\overline{(u_{0})}_{ S }||^{\frac{2}{p}}_{L^{2}( S )} \left(\frac{1}{t}\right)^{\frac{1}{p}},
\end{align}
for all $u_{0} \in L^{2}( S )$ and $t \in (0,\infty)$, where $C\geq 0$ is a constant (being determined explicitly later) depending only on $p$, $ S $ and $\gamma$. Actually, it will turn out that (\ref{introbound_L1}) is a corollary of a slightly stronger result which is more technical to formulate and will be postponed until Section \ref{expbounds}.\\
Moreover, it will be shown that even
\begin{align}
\label{introbound_sup}
||T(t)u_{0}-\overline{(u_{0})}_{ S }||_{L^{\infty}( S )} \leq  \hat{C} ||u_{0}-\overline{(u_{0})}_{ S }||^{\frac{2}{p}}_{L^{2}( S )} \left(\frac{1}{t}\right)^{\frac{1}{p}},
\end{align}
for all $t \in (0,\infty)$ and $u_{0} \in L^{p}( S )$, if $p$ is sufficiently larger than $n$, where $\hat{C}\geq 0$ is a constant (being determined explicitly later) depending only on $p$, $ S $ and $\gamma$. (Hereby "sufficiently" depends on the integrability of $\gamma$.)\\
Additionally, an extinction principle will be proven, i.e. if $p$ is sufficiently smaller than $n$ and if $u_{0}\in L^{2}( S )$, then there is a finite time $T^{\ast}$ such that $T(t)u_{0}$ is constantly the average of the initial value for any $t \geq T^{\ast}$. It will actually be possible to give an explicit formula for $T^{\ast}$.\\
Finally, these results show that, if $\gamma$ is sufficiently integrable and $n=2$ then the solutions extinct after finite time if $p \in (1,2)$ and $u_{0} \in L^{2}( S )$; and (\ref{introbound_sup}) holds if $p \in (2,\infty)$ and $u_{0} \in L^{p}( S )$.\\
Note that the considered initial value problem can be used to model the evolution of a fluvial landscape. Consequently, in this application one always has $n=2$ and $u_{0} \in L^{\infty}( S ) \subseteq L^{2}( S ) \cap L^{p}( S )$.\\ 

Before proceeding with a detailed derivation of all these results, some words on the literature are in order. Firstly, the monograph \cite{acmbook} by F. Andreu, V. Caselles and J.M. Maz\'{o}n deals with existence, uniqueness, asymptotic and qualitative results for many initial value problems. Even though the initial value problem considered here is not considered in this book, the asymptotic results there, served as an inspiration for the current paper.\\
Moreover, the monograph \cite{BenilanBook} by P. B\'{e}nilan, M. Crandall and A. Pazy is a detailed and comprehensive introduction to the general theory of nonlinear semigroups and evolution equation.

\section{Assumptions and preliminary results}
\label{basicdef}
Some notational preliminaries are in order: For any $m$-dimensional Borel measurable set $\Omega$, where $m \in \mathbb{N}$, $\mathfrak{B}(\Omega)$ denotes the Borel $\sigma$-algebra on this set. Moreover, if $\mu:\mathfrak{B}(\Omega) \rightarrow [0,\infty]$ is a measure and $q \in [1,\infty]$ then $L^{q}(\Omega,\mu;\mathbb{R}^{m})$ denotes the usual Lebesgue spaces and $||\cdot||_{L^{q}(\Omega,\mu;\mathbb{R}^{m})}$ denotes the canonical norm on these spaces.\\ 
If $m=1$ then $L^{q}(\Omega,\mu;\mathbb{R}^{m})$ will be abbreviated by $L^{q}(\Omega,\mu)$ and if $\mu$  is the Lebesgue measure then $L^{q}(\Omega)$ will be written. Of course the analogous convention applies to $||\cdot||_{L^{q}(\Omega,\mu;\mathbb{R}^{m})}$.\\
In addition, $L^{1}_{\text{Loc}}(\Omega)$, $L^{1}_{\text{Loc}}(\Omega;\mathbb{R}^{m})$ denote, if $\Omega$ is open, the space of locally Lebesgue integrable functions $f:\Omega\rightarrow \mathbb{R}$, $f:\Omega\rightarrow \mathbb{R}^{m}$ respectively.\\
Moreover, if $\Omega$ is open then $W^{1,1}_{\text{Loc}}(\Omega)$ denotes the space of weakly differentiable functions and $\nabla f$ denotes the weak derivative of any $f \in W^{1,1}_{\text{Loc}}(\Omega)$. In addition, $W^{1,q}(\Omega)$ denotes the Sobolev space of once weakly differentiable functions, such that the function and all of its weak derivatives are in $L^{q}(B)$.\\

If $(X,||\cdot||_{X})$ is a Banach space, then $W^{1,1}_{\text{Loc}}( (0,\infty);X)$ denotes the space of all functions $f:(0,\infty)\rightarrow X$ which are locally absolutely continuous and differentiable a.e. For an $f \in W^{1,1}_{\text{Loc}}((0,\infty);X) $ the function $f^{\prime}$ denotes the almost everywhere existing derivative of $f$. Moreover, $C( [0,\infty);X)$ denotes the space of all continuous functions $f:[0,\infty)\rightarrow X$ and $2^X$ denotes the power set of $X$. \\
If in addition $a \in (0,\infty)$ then
\begin{align*}
L^{1}([0,a];X):=\{f:[0,a]\rightarrow X|~f\text{ is strongly meas. and } \int \limits_{0} \limits^{a} ||f(t)||_{X}dt <\infty\}.
\end{align*} 
Let $B:X \rightarrow 2^{X}$ be a multi-valued operator, then its graph $G(B) \subseteq X \times X$ is defined by $G(B):=\{(x,\hat{x}):~\hat{x} \in Bx\}$. Moreover, it is clear that any set $\tilde{B} \subseteq X\times X$ uniquely defines an operator $B$, by $\hat{x} \in Bx$ if and only if $(x,\hat{x}) \in \tilde{B}$. Therefore, an operator and its graph will be denoted by the same latter.\\
Moreover, the domain of $B$ is defined by $D(B):=\{x \in X: Bx \neq \emptyset \}$ and $B$ is called single-valued, if $Bx$ contains precisely one element for any $x \in D(B)$. If $B$ is single valued, then the set $Bx$, containing only the element $\hat{x}$, is identified with this element, for any $x \in D(B)$.\\ 

Moreover, $\lambda$ denotes the Lebesgue measure and $|\cdot|$ the euclidean norm on $\mathbb{R}^{m}$. In addition, the canonical inner product of any $x,y\in \mathbb{R}^{m}$ is denoted by $x\cdot y$.\\

Finally, $A_{q}(\mathbb{R}^{m})$ denotes, for any $q \in (1,\infty)$, the class of \textit{Muckenhoupt weights}, i.e. $A_{q}(\mathbb{R}^{m})$ consists of all functions $\gamma_{0}:\mathbb{R}^{m}\rightarrow \mathbb{R}$ such that $\gamma_{0}>0$ a.e., $\gamma_{0} \in L^{1}_{\text{Loc}}(\mathbb{R}^{m})$ and
\begin{align*}
\sup \limits_{\substack{B \subseteq \mathbb{R}^{n}\\ B \text{ is a ball}}}\Big[ \frac{1}{\lambda(B)} \int \limits_{B} \gamma_{0}d \lambda \Big( \frac{1}{\lambda(B)} \int \limits_{B} \gamma_{0}^{\frac{1}{1-q}}d \lambda \Big)^{q-1}\Big]< \infty.
\end{align*} 

Now the assumptions on the quantities $ S $, $\gamma$ and $p$ mentioned in the introduction, will be made precise.\\
Here and in everything that follows let $n \in \mathbb{N}\setminus \{1\}$ and $\emptyset \neq  S  \subseteq \mathbb{R}^{n}$ be a non-empty, open, connected and bounded sets of class $C^{1}$. \\
Moreover, let $p \in (1,\infty)\setminus \{2\}$. We are not interested in the linear case $p=2$; particularly the regularization effect (see\cite{cao} Theorem 4.4), which is needed in the present paper, is not applicable if $p=2$. Therefore, this value for $p$ is excluded.\\ 
Additionally, let $\gamma: S   \rightarrow (0,\infty)$ be such that $\gamma \in L^{\infty}( S )$, $\gamma^{\frac{1}{1-p}} \in L^{1}( S)$ and assume that there is a $\gamma_{0} \in A_{p}(\mathbb{R}^{n})$ such that $\gamma_{0}|_{ S }=\gamma$ a.e. on $S$.\\
Furthermore, let $\nu : \mathfrak{B}( S ) \rightarrow [0,\infty)$ be the measure induced by $\gamma$, i.e. \linebreak$\nu(B):= \int \limits_{B}\gamma d \lambda$ for all $B \in \mathfrak{B}( S )$ and introduce the weighted Sobolev space
\begin{align*}
W_{\gamma }^{1,p}( S ):=\{f \in L^{p}( S ): ~\nabla f \in L^{p}( S ,\nu;\mathbb{R}^{n})\}. 
\end{align*} 
Now introduce $J_{0}$ as the space of all convex, lower semi-continuous functions $j:\mathbb{R}\rightarrow [0,\infty]$ fulfilling $j(0)=0$. Given $f,h \in L^{1}( S )$, one writes $f<<h$ whenever
\begin{align*}
\int \limits_{ S }j \circ f d\lambda \leq \int \limits_{ S }  j \circ h d \lambda,~\forall j \in J_{0}.
\end{align*}  
Moreover, an operator $B\subseteq L^{1}( S ) \times L^{1}( S )$ is called \textit{completely accretive} if $f-h<<f-h+\alpha (\hat{f}-\hat{h})$, for all $(f,\hat{f}),~(h,\hat{h})\in B$ and $\alpha \in (0,\infty)$. The reader is referred to \cite{cao} for a detailed discussion of the concept of complete accretivity.

\begin{rem}\label{generaliedwdlemma} In the sequel, $\tau_{k}:\mathbb{R}\rightarrow \mathbb{R}$, where $k \in (0,\infty)$, denotes the standard truncation function, i.e. $\tau_{k}(s):= s$, if $|s|< k$ and $\tau_{k}(s):= k\sgn(s)$, if $|s|\geq k$.  Moreover, if $f:S \rightarrow \mathbb{R}$ is Borel measurable and fulfills \linebreak$\tau_{k}(f) \in W^{1,1}_{\text{Loc}}(S)$ for all $k \in (0,\infty)$, then $\tilde{\nabla }f:S\rightarrow \mathbb{R}^{n}$, denotes the (up to equality a.e.) uniquely determined function fulfilling
\begin{align}
\label{generaliedwdtau}
\nabla \tau_{k}(f)=\tilde{\nabla }f\id_{\{|f|<k\}},~\forall k \in (0,\infty) 
\end{align}
a.e. on $S$. The function $\tilde{\nabla }f$ is called the \textit{generalized weak derivative} of $f$. Note that if $f:S\rightarrow \mathbb{R}$ is generalized weakly differentiable, then $f\in W^{1,1}_{\text{Loc}}(S)$ if and only if  $\tilde{\nabla }f \in L^{1}_{\text{Loc}}(S;\mathbb{R}^{n})$; and in this case $\tilde{\nabla }f=\nabla f$. Cf. \cite{gwd}, for these and further properties.
\end{rem}

The following operators are considered in \cite{main} to show that (\ref{eveq}) admits a unique entropy solution.

\begin{defn}\label{defo} Let $A \subseteq L^{1}( S ) \times L^{1}( S )$ be defined by: $(f,\hat{f}) \in A$ if and only if the following assertions hold.
	\begin{enumerate}
		\item $f \in W^{1,p}_{\gamma}( S ) \cap L^{\infty}( S )$. 
		\item $\hat{f} \in L^{1}( S )$.
		\item $\int \limits_{ S}  \gamma|\nabla f|^{p-2}\nabla f\cdot\nabla \varphi  d \lambda = \int \limits_{ S } \hat{f} \varphi d \lambda$ for all $\varphi\in W^{1,p}_{\gamma }( S )\cap L^{\infty}( S )$.
	\end{enumerate}
Moreover, let $\mathcal{A} \subseteq L^{1}( S ) \times L^{1}( S )$ be defined by: $(f,\hat{f})\in \mathcal{A}$ if and only if the following assertions hold.
\begin{enumerate}\setcounter{enumi}{3}
	\item $f,\hat{f}\in L^{1}( S )$.
	\item $\tau_{k}(f)\in W^{1,p}_{\gamma}( S )$ for all $k\in (0,\infty)$ .
	\item $\int \limits_{ S }\gamma  |\tilde{\nabla }f|^{p-2} \tilde{\nabla }f\cdot\nabla (\tau_{k}(f-\varphi))d \lambda \leq \int \limits_{ S }\hat{f}\tau_{k}(f-\varphi)d\lambda$ for all $k\in (0,\infty)$ and $\varphi \in W^{1,p}_{\gamma}( S )\cap L^{\infty}( S )$.
\end{enumerate}
\end{defn}

Finally, for the reader's convenience, the following result will be extracted from \cite{main}, Section 3. This existence and uniqueness result is fundamental for that what follows.

\begin{thm}\label{mainold} $\mathcal{A}$  is completely accretive, m-accretive and the closure of $A$. Moreover, $D(A)$ is a dense subset of $(L^{1}( S ),||\cdot||_{L^{1}( S )})$. Consequently, the evolution equation
	\begin{align}
	\label{mainoldeq}
	0 \in u^{\prime}(t)+\mathcal{A}u(t) \text{ for a.e. }t\in (0,\infty) \text{ and } u(0)=u_{0}
	\end{align}
	has for a given $u_{0} \in L^{1}( S )$ precisely one mild solution. Moreover, this mild solution is also the unique strong solution. Hence, there is a semigroup $(T(t))_{t\geq 0}$, with $T(t):L^{1}( S )\rightarrow L^{1}( S )$, fulfilling 
	\begin{align*}
	T(\cdot)u_{0} \in C([0,\infty);L^1(S)) \cap  W^{1,1}_{\text{Loc}}((0,\infty);L^{1}( S ))
	\end{align*}
	and
	\begin{align*}
	0 \in T^{\prime}(t)u_{0}+\mathcal{A}T(t)u_{0}\text{, }T(t)u_{0} \in D(\mathcal{A})\text{ a.e. } t \in (0,\infty) \text{ and } T(0)u_{0}=u_{0}
	\end{align*}
	for all $u_{0} \in L^{1}(S)$.
\end{thm}

In what follows, $(T(t))_{t \geq 0}$ denotes the strongly continuous semigroup introduced in Theorem \ref{mainold} and $T^{\prime}(\cdot)u_{0}$ denotes, for any $u_{0} \in L^{1}(S)$, the derivative of $T(\cdot)u_{0}$, which exists almost everywhere on $(0,\infty)$. Note that the null-set on which $T(\cdot)u_{0}$ is not differentiable depends on $u_{0}$.\\
In the following sections, initial values are simply denoted by $u$, $v$, etc. and no longer by $u_{0}$, $v_{0}$, etc.

\section{Conservation of mass and other basic properties}
\label{com} 
The purpose of this section is to derive some basic properties of $(T(t))_{ t \geq 0 }$ among them, the conservation of mass principle.\\

For any $u \in L^{1}( S )$, let $\overline{(u)}_{ S }$ denote its \textit{average}, i.e. $\overline{(u)}_{ S }:= \frac{1}{\lambda( S )} \int \limits_{ S } u d \lambda$. By slightly abusing notation, the constant function mapping from $ S $ to $\mathbb{R}$, which takes only the value $\overline{(u)}_{ S }$  will also be denoted by  $\overline{(u)}_{ S }$.\\

\begin{lem}\label{domcharlemma} $A$ is single-valued. Moreover, if $f \in D(\mathcal{A})\cap L^{\infty}( S )$ and $\hat{f} \in \mathcal{A}f$, then $f \in D(A)$ and $\hat{f}=Af$.
\end{lem}
\begin{proof} It is plain that $A$ is single-valued, since $(f,\hat{f}),~(f,\tilde{f})\in A$ implies
	\begin{align*}
	\int \limits_{S}(\hat{f}-\tilde{f})\varphi d \lambda = 0,~\forall \varphi \in W^{1,p}_{\gamma}(S)\cap L^{\infty}(S).
	\end{align*}
	Now let $f  \in D(\mathcal{A})\cap L^{\infty}( S )$ and $\hat{f} \in \mathcal{A}f$, then $\tau_{k}(f) \in W^{1,p}_{\gamma}( S )$  for all $k \in (0,\infty)$. Consequently, $f \in W^{1,p}_{\gamma}( S ) \cap L^{\infty}( S )$ by choosing $k>||f||_{L^{\infty}( S )}$. Hence the claim follows if 
	\begin{align}
	\label{eveqlemmaproof3}
	\int \limits_{ S }  \gamma|\nabla f|^{p-2}\nabla f\cdot\nabla \varphi d\lambda= \int \limits_{ S } \hat{f} \varphi d \lambda,~\forall \varphi \in W^{1,p}_{\gamma}( S ) \cap L^{\infty}( S ).
	\end{align}
	Proof of (\ref{eveqlemmaproof3}). It follows from the definition of $\mathcal{A}$ that 
	\begin{align*}
	\int \limits_{ S }\gamma  |\tilde{\nabla }f|^{p-2} \tilde{\nabla }f\cdot\nabla (\tau_{k}(f-\varphi))d \lambda \leq \int \limits_{ S }\hat{f}\tau_{k}(f-\varphi)d\lambda,
	\end{align*}
	for all $\varphi \in W^{1,p}_{\gamma}( S )\cap L^{\infty}( S )$ and $k \in (0,\infty)$.\\
	Observe that $f \in W^{1,p}_{\gamma}( S ) \cap L^{\infty}( S )$ implies $\tilde{\nabla }f=\nabla f$ on $ S $ (see Remark \ref{generaliedwdlemma}) and that $\varphi=f-\tilde{\varphi}$, where $\tilde{\varphi} \in W^{1,p}_{\gamma}( S )\cap L^{\infty}( S )$, is a valid choice as a test function in the previous equation, hence
	\begin{align}
	\label{eveqlemmaproof5}
	\int \limits_{ S }\gamma  |\nabla f|^{p-2} \nabla f\cdot\nabla (\tau_{k}(\tilde{\varphi}))d \lambda \leq \int \limits_{ S }\hat{f}\tau_{k}(\tilde{\varphi})d\lambda,
	\end{align}
	for all $\tilde{\varphi} \in W^{1,p}_{\gamma}( S )\cap L^{\infty}( S )$ and $k \in (0,\infty)$.\\
	Now (\ref{eveqlemmaproof5}) yields, by choosing  $k>||\tilde{\varphi}||_{L^{\infty}( S )}$ for a given $\tilde{\varphi} \in W^{1,p}_{\gamma}( S )\cap L^{\infty}( S )$, that
	\begin{align}
	\label{eveqlemmaproof6}
	\int \limits_{ S }\gamma  |\nabla f|^{p-2} \nabla f\cdot\nabla \tilde{\varphi}d \lambda \leq \int \limits_{ S }\hat{f}\tilde{\varphi}d\lambda,~\forall\tilde{\varphi} \in W^{1,p}_{\gamma}( S )\cap L^{\infty}( S ).
	\end{align}
	Conclusively the claim follows since $\tilde{\varphi}$ can be replaced by $-\tilde{\varphi}$ as a test function in (\ref{eveqlemmaproof6}).
\end{proof}

\begin{rem}\label{infcontractionrema} As it turns out, the preceding lemma is only useful for our purposes if one can show the following: If $v\in L^{1}( S )$, $w \in L^{\infty}( S )$ and $v << w$, then $||v||_{L^{\infty}( S )} \leq ||w||_{L^{\infty}( S )}$.\\
	In fact, choosing $j(x):=\max(|x|-||w||_{L^{\infty}( S )},0)$, $x\in \mathbb{R}$, yields
	\begin{align*}
	0 \leq \int \limits_{ S } \max(|v|-||w||_{L^{\infty}( S )},0)d \lambda \leq \int \limits_{ S } \max(|w|-||w||_{L^{\infty}( S )},0)d \lambda=0,
	\end{align*}
	if $ v<<w $ and consequently $|v|\leq||w||_{L^{\infty}( S )}$ a.e. on $ S $.  
\end{rem}

\begin{lem}\label{nicealemma} The following assertions hold. 
	\begin{enumerate}
		\item \label{l}  $T(t)u-T(t)v<<u-v$ for all $u,~v \in L^{1}( S )$ and  $t \in [0,\infty)$.
		\item  $T(t)u<<u$ for all $u \in L^{1}( S )$ and  $t \in [0,\infty)$.
		\item $||T(t)w||_{L^{\infty}( S )} \leq ||w||_{L^{\infty}( S )}$ for every  $w \in L^{\infty}( S )$ and every $t \in [0,\infty)$.
		\item $T(t)w \in D(A)$ and $-T^{\prime}(t)w=AT(t)w$ for every  $w \in L^{\infty}( S )$ and almost every $t \in (0,\infty)$.
	\end{enumerate}
\end{lem}
\begin{proof} The first assertion follows from \cite{cao}, Prop. 4.1. Moreover, it is plain that $0 \in D(A)$ and $A0=0$ which clearly implies $0 \in D(\mathcal{A})$ and $0 \in \mathcal{A}0$. This yields that $T(t)(0)=0$ for all $t \in [0,\infty)$. Consequently, the second assertions holds as  well.\\
	The third assertion follows by combining the second and Remark \ref{infcontractionrema}.\\
	Finally, Theorem \ref{mainold}, Lemma \ref{domcharlemma} and the third assertion yield the fourth.
\end{proof}

\begin{lem}\label{comlemma1} Let $u \in L^{1}( S )$, then $\overline{(T(t)u)}_S=\overline{(u)}_S$ for every $t \geq 0$.
\end{lem}
\begin{proof} Firstly, Lemma \ref{nicealemma} yields that it suffices to prove the claim for \linebreak$u \in D( A )$, since this is according to Theorem \ref{mainold} a dense subset of \linebreak$(L^{1}(S)||\cdot||_{L^{1}(S)})$.
	So let $u \in D(A)$ be given. Moreover, introduce $\tau \in (0,\infty)$ and $f:[0,\tau]\rightarrow \mathbb{R}$ by $f(t):=\int \limits_{S}T(t)ud\lambda$, for all $t \in [0,\tau]$.\\
	It follows from \cite{BenilanBook}, Lemma 7.8 that $(T(\cdot)u)|_{[0,\tau]}$ is Lipschitz continuous which obviously implies that $f$ is Lipschitz continuous as well. Moreover, it is plain that $f^{\prime}(t)=\int \limits_{S}T^{\prime}(t)ud\lambda$.\\
	In addition, note that $D(A)\subseteq L^{\infty}(S)$ which yields by the aid of Lemma \ref{nicealemma} that
	\begin{align*}
	f^{\prime}(t) = - \int \limits_{S}\gamma |\nabla T(t)u|^{p-2}\nabla T(t)u\cdot\nabla \varphi d\lambda = 0,
	\end{align*}
	where $\varphi:S\rightarrow \mathbb{R}$ denotes the function which is constantly one.\\
	Consequently, $f$ is constant and therefore $\overline{(u)}=\overline{(T(t)u)}$ for all $t \in [0,\tau]$ which gives the claim as $\tau$ is arbitrary.
\end{proof}

\section{Upper bounds and asymptotic results}
\label{expbounds}
The purpose of this section is to prove the results (\ref{introbound_conv}), (\ref{introbound_L1}) and (\ref{introbound_sup}) mentioned in the introduction. Actually, it will turn out that (\ref{introbound_L1}) is a corollary of a slightly stronger result.

\begin{lem}\label{additionlemma} Let $u \in L^{1}( S )$ and $\varphi: S  \rightarrow \mathbb{R}$ be a constant function. Then
	\begin{align}
	\label{additivitylemmaeq}
	T(t)(u+\varphi)=T(t)(u)+\varphi,~\forall t \in [0,\infty).
	\end{align}
	Consequently, if $T(\cdot)u$ is differentiable in $t \in (0,\infty)$, then $T(\cdot)(u+\varphi)$ is differentiable in $t$ and $T^{\prime}(t)(u+\varphi)=T^{\prime}(t)u$.
\end{lem}
\begin{proof} Let $u \in L^{\infty}( S )$, let  $\varphi: S  \rightarrow \mathbb{R}$ be a constant function and introduce $f:[0,\infty)\rightarrow L^{1}( S )$ by $f(t):=T(t)(u)+\varphi$.\\
	It is clear that $f(0)=u+\varphi$ and also that $f$ is continuous on $[0,\infty)$ and an element of $W^{1,1}_{\text{Loc}}((0,\infty);L^{1}( S ))$, since $T(\cdot)u$ has these properties.\\
	Now observe that obviously $f^{\prime}(t)=T^{\prime}(t)u$ for a.e. $t \in (0,\infty)$. Moreover, one has for any $\varphi \in W^{1,p}_{\gamma}( S ) \cap L^{\infty}( S )$ that
	\begin{align*}
	\int \limits_{ S }  \gamma|\nabla f(t)|^{p-2}\nabla f(t)\cdot\nabla \varphi  d \lambda = \int \limits_{ S }  \gamma|\nabla T(t)u|^{p-2}\nabla T(t)u\cdot\nabla \varphi  d \lambda
	\end{align*}
	which implies, together with $f^{\prime}(t)=T^{\prime}(t)u$ for a.e. $t \in (0,\infty)$ and Lemma \ref{nicealemma}, that $f(t) \in D(A)$ and $-f^{\prime}(t) = Af(t)$ for a.e. $t \in (0,\infty)$. Consequently the claim is verified for initial values $u \in L^{\infty}( S )$.\\
	Conclusively, applying Lemma \ref{nicealemma} yields that (\ref{additivitylemmaeq}) holds also for arbitrary initial values $u \in L^{1}( S )$, since  $L^{\infty}( S )$ is dense in $(L^{1}( S ),||\cdot||_{L^{1}( S )})$.\\ 
	Finally observe that (\ref{additivitylemmaeq}) clearly implies the remaining part of the claim. 
\end{proof}

\begin{rem} In everything which follows let $p_{0} \in [1,p]$ be the constant defined by
	\begin{align*}
	p_{0}:= \inf \{q > 1: \gamma^{\frac{1}{1-q}} \in L^{1}( S )\}.
	\end{align*}
	Since $\gamma^{\frac{1}{1-p}} \in L^{1}( S )$ by assumption it is clear that indeed $p_{0} \leq p$.\\
	The following lemma reveals that even $p_{0} <p$.	
\end{rem}

\begin{lem}\label{p0lemma} If $q>p_{0}$ then $\gamma^{\frac{1}{1-q}} \in L^{1}( S )$. Moreover, $p_{0} < p$.
\end{lem}
\begin{proof} Let $q>p_{0}$, then there is $\tilde{q} \in [p_{0},q)\setminus\{1\}$ such that $\gamma^{\frac{1}{1-\tilde{q}}} \in L^{1}( S )$. Since trivially $\frac{1-q}{1-\tilde{q}}>1$, H\"older's inequality yields
	\begin{align*}
	\int \limits_{ S } \gamma^{\frac{1}{1-q}}d \lambda \leq \lambda( S )^{\frac{\tilde{q}-q}{1-q}} \left( \int \limits_{ S } \gamma^{\frac{1}{1-\tilde{q}}} d \lambda \right)^{\frac{1-\tilde{q}}{1-q}} < \infty,
	\end{align*}
	which implies $\gamma^{\frac{1}{1-q}} \in L^{1}( S )$.\\
	By assumption there is $\gamma_{0} \in A_{p}(\mathbb{R}^{n})$  such that $\gamma=\gamma_{0}$ a.e. on $ S $. Moreover, there is an $\varepsilon \in (0,p-1)$ such that $\gamma_{0} \in A_{p-\varepsilon}(\mathbb{R}^{n})$. (See \cite{mw}, Ch. IX Prop. 4.3 and Theorem 5.5.)\\
	Since $ S $ is bounded, there is a ball $B \subseteq \mathbb{R}^{n}$ containing $ S $ which implies $\gamma_{0}^{\frac{1}{1-(p-\varepsilon)}} \in L^{1}( S )$. This implies $p_{0}<p$, since $\gamma=\gamma_{0}$ a.e. on $ S $.
\end{proof}

\begin{lem}\label{technicallemma1} Let $0 \leq \delta < \frac{p-p_{0}}{p_{0}}$ and $f \in W^{1,p}_{\gamma}( S )$ then $f \in W^{1,1+\delta}( S )$ and 
	\begin{align}
	\label{smartdduestimate}
	||\nabla f||_{ L^{1+\delta}( S ;\mathbb{R}^{n})} \leq \left( \int \limits_{ S } \gamma^{\frac{1+\delta}{1+\delta-p}} d \lambda \right)^{\frac{p-1-\delta}{p(1+\delta)}} ||\nabla f||_{L^{p}( S ,\nu;\mathbb{R}^{n})}<\infty.
	\end{align}
\end{lem}
\begin{proof} Let $0 \leq \delta < \frac{p-p_{0}}{p_{0}}$. (Note that $p_{0} < p$, consequently such a $\delta$ does indeed exists.)\\
	Let $f \in W^{1,p}_{\gamma}( S )$, then obviously $f \in W^{1,1}_{\text{Loc}}( S )$ as well as $f \in L^{p}( S )$.\\  
	Moreover, note that $1+\delta < 1+\frac{p-p_{0}}{p_{0}} \leq p$ and consequently $f \in L^{1+\delta}( S )$, since $\lambda( S )<\infty$.\\
	Conclusively the claim follows once (\ref{smartdduestimate}) is proven.\\
	First of all $1+\delta - p \neq 0$.\\
	Secondly observe that $\frac{p}{1+\delta}>p_{0}$, consequently Lemma \ref{p0lemma} yields
	\begin{align}
	\label{smartdduestimate1}
	\int \limits_{ S } \gamma^{\frac{1+\delta}{1+\delta-p}} d \lambda = \int \limits_{ S } \gamma^{\frac{1}{1-\frac{p}{1+\delta}}} d \lambda <\infty.
	\end{align}
	Finally, (\ref{smartdduestimate}) follows from the following estimate, where H\"older's inequality is used.
	\begin{align*}
		||\nabla f||_{L^{1+\delta}( S ;\mathbb{R}^{n})}
		& =  \left( \int \limits_{ S }  |\nabla f|^{1+\delta} \gamma^{\frac{1+\delta}{p}}\gamma^{-\frac{1+\delta}{p}}	d\lambda \right)^{\frac{1}{1+\delta}}\\
		& \leq  \left( \left( \int \limits_{ S }  |\nabla f|^{p} \gamma d\lambda \right)^{\frac{1+\delta}{p}} \left( \int \limits_{ S } \gamma^{\frac{1+\delta}{1+\delta-p}}	d\lambda \right)^{\frac{p-1-\delta}{p}}\right)^{\frac{1}{1+\delta}}\\
		& =   ||\nabla f||_{L^{p}( S ,\nu;\mathbb{R}^{n})} \left( \int \limits_{ S } \gamma^{\frac{1+\delta}{1+\delta-p}}	d\lambda \right)^{\frac{p-1-\delta}{p(1+\delta)}},
	\end{align*}
which is finite due to (\ref{smartdduestimate1}).
\end{proof}

The preceding lemma is a slight modification of \cite{sobweightedcap}, Prop. 2.1. There, an analogues result is proven for Sobolev spaces, where the function and its weak derivative need to be integrable with respect to the same measure and not to different ones as in our setting.

\begin{lem}\label{techniacallemma2} Let $u \in L^{2}( S ) \cap L^{p}( S )$, then $T(t)u \in L^{2}( S ) \cap W^{1,p}_{\gamma}( S )$ for a.e. $t \in (0,\infty)$ and moreover
	\begin{align}
	\label{upperbounderivative}
	||\nabla T(t)u||_{L^{p}( S ,\nu;\mathbb{R}^{n})} \leq \left( \frac{2}{|p-2|} \right)^{\frac{1}{p}} ||u-\overline{(u)}_{ S }||^{\frac{2}{p}}_{L^{2}( S )} \left(\frac{1}{t}\right)^{\frac{1}{p}} 
	\end{align}
	for a.e. $t \in (0,\infty)$.
\end{lem}
\begin{proof} Let $t \in (0,\infty)$ be such that  $0 \in T^{\prime}(t)u+\mathcal{A}T(t)u$. Theorem \ref{mainold} implies that almost every value in $(0,\infty)$ is a valid choice for $t$.\\
	Let $u \in L^{2}( S )\cap L^{p}( S )$ then Lemma \ref{nicealemma} yields $T(t)u \in L^{2}( S )\cap L^{p}( S )$.\\ 
	Note that $T(t)u$ is generalized weakly differentiable. Consequently, if one proves that 
	\begin{align}
	\label{upperbounderivativeproof1}
	\int \limits_{ S }\gamma|\tilde{\nabla }T(t)u|^{p}d\lambda \leq  \frac{2}{|p-2|}  ||u-\overline{(u)}_{ S }||^{2}_{L^{2}( S )} \frac{1}{t},
	\end{align}
	then obviously $\tilde{\nabla }T(t)u \in L^{p}( S ,\nu;\mathbb{R}^{n}) \subseteq L^{1}( S ;\mathbb{R}^{n})$ and therefore, by virtue of Remark \ref{generaliedwdlemma}, $\tilde{\nabla }T(t)u=\nabla T(t)u$ a.e. on $ S $.\\
	Hence, if (\ref{upperbounderivativeproof1}) holds, then also (\ref{upperbounderivative}) as well as $T(t)u \in L^{2}( S ) \cap W^{1,p}_{\gamma}( S )$.\\
	Proof of (\ref{upperbounderivativeproof1}). First of all observe that
	\begin{align*}
	\mathcal{A}(\alpha v)=\alpha^{p-1}\mathcal{A} v,~\forall v \in D(\mathcal{A}),~\alpha \in (0,\infty).
	\end{align*}
	Consequently Theorem \ref{mainold} together with \cite{cao}, Theorem 4.4 yield
	\begin{align}
	\label{upperbounderivativeproof2}
	||T^{\prime}(t)(u-\overline{(u)}_{ S })||_{L^{2}( S )} \leq \frac{2}{|p-2|t} ||u-\overline{(u)}_{ S }||_{L^{2}( S )}.
	\end{align}
	Moreover, one infers from Fatou's lemma and Lemma \ref{additionlemma} that
	\begin{align*}
	\int \limits_{ S }\gamma|\tilde{\nabla }T(t)u|^{p}d\lambda \leq \liminf \limits_{k \rightarrow \infty}\int \limits_{ S } -T^{\prime}(t)\left(u-\overline{(u)}_{ S }\right) \tau_{k}\left(T(t)u-\overline{(u)}_{ S }\right) d\lambda
	\end{align*}
	Consequently Cauchy Schwarz inequality, (\ref{upperbounderivativeproof2}) and Lebesgue's theorem yield
	\begin{align*}
	\int \limits_{ S }\gamma|\tilde{\nabla }T(t)u|^{p}d\lambda \leq  \frac{2}{|p-2|t} ||u-\overline{(u)}_{ S }||_{L^{2}( S )} ||T(t)u-\overline{(u)}_{ S }||_{L^{2}( S )}
	\end{align*}
	Finally, (\ref{upperbounderivativeproof1}) follows by applying (\ref{additivitylemmaeq}) and  Lemma \ref{nicealemma}.
\end{proof}

Here, and in everything that follows $C_{ S ,q}$ denotes the Poincar\'{e} constant of $ S $, for any $q \in [1,\infty)$, i.e.
$C_{ S ,q} \in (0,\infty)$ is the smallest constant depending only on $ S $ and $q$, such that
\begin{align*}
||f-\overline{(f)}_{ S }||_{L^{q}( S )} \leq C_{ S ,q}||\nabla f||_{L^{q}( S ;\mathbb{R}^{n})},~\forall f \in W^{1,q}( S ).
\end{align*}
Note that $ S $ is assumed to be open, bounded, connected and of class $C^{1}$. Consequently  the Poincar\'{e} inequality implies the existence of $C_{ S ,q}$.

\begin{thm}\label{mainwithoutasumptions} Let $0 \leq \delta < \frac{p-p_{0}}{p_{0}}$ and $u \in L^{2}( S )\cap L^{1+\delta}( S )$, then 
	\begin{align}
	\label{ubtheoremeq1}
	||T(t)u-\overline{(u)}_{ S }||_{L^{1+\delta}( S )} \leq C_{ S ,1+\delta}\Gamma_{\delta,p}  ||u-\overline{(u)}_{ S }||^{\frac{2}{p}}_{L^{2}( S )} \left(\frac{1}{t}\right)^{\frac{1}{p}}
	\end{align}
	for every $t \in (0,\infty)$, where
	\begin{align}
	\label{ubtheoremeq3}
		\Gamma_{\delta,p}:=\left( \int \limits_{ S } \gamma^{\frac{1+\delta}{1+\delta-p}} d \lambda \right)^{\frac{p-1-\delta}{p(1+\delta)}} \left( \frac{2}{|p-2|} \right)^{\frac{1}{p}}<\infty.
	\end{align}
\end{thm}
\begin{proof} Let $0 \leq \delta < \frac{p-p_{0}}{p_{0}}$ and $u \in L^{2}( S )\cap L^{p}( S )$.\\
	Let $t \in (0,\infty)$ be such that the assertions of Lemma \ref{techniacallemma2} hold. Since \linebreak$T(t)u \in W^{1,p}_{\gamma}( S )$ Lemma \ref{technicallemma1} yields $T(t)u \in W^{1,1+\delta}( S )$ and consequently Lemma \ref{comlemma1}, Poincar\'{e}'s inequality, (\ref{smartdduestimate}) and (\ref{upperbounderivative}) imply
	\begin{align*}
		||T(t)u-\overline{(u)}_{ S }||_{L^{1+\delta}( S )}
		& =  ||T(t)u-\overline{(T(t)u)}_{ S }||_{L^{1+\delta}( S )} \\
		& \leq  C_{ S ,1+\delta}||\nabla T(t)u||_{L^{1+\delta}( S ;\mathbb{R}^{n})} \\
		& \leq  C_{ S ,1+\delta}\Gamma_{\delta,p} ||u-\overline{(u)}_{ S }||^{\frac{2}{p}}_{L^{2}( S )} \left(\frac{1}{t}\right)^{\frac{1}{p}}, 
	\end{align*}
	i.e. (\ref{ubtheoremeq1}) holds for $u \in L^{2}( S )\cap L^{p}( S )$ and almost every $t \in (0,\infty)$.\\
	
	Now let $t \in (0,\infty)$ be arbitrary and still assume  $u \in L^{2}( S )\cap L^{p}( S )$.\\
	Moreover, let $(t_{m})_{m \in  \mathbb{N}} \subseteq (0,\infty)$ be such that $\lim \limits_{m \rightarrow \infty} t_{m}=t$ and assume that (\ref{ubtheoremeq1}) holds for each $t_{m}$.\\
	Since $T(\cdot)u:[0,\infty)\rightarrow (L^{1}( S ),||\cdot||_{L^{1}( S )})$ is continuous one obtains (by passing to a subsequence if necessary) that $\lim \limits_{m \rightarrow \infty} T(t_{m})u-\overline{(u)}_{ S }=T(t)u-\overline{(u)}_{ S }$ a.e. on $ S $. Consequently one infers by virtue of Fatou's Lemma that 
	\begin{align*}
	||T(t)u-\overline{(u)}_{ S }||_{L^{1+\delta}( S )} 
	& =  \left( \int \limits_{ S } \lim \limits_{m \rightarrow \infty} |T(t_{m})u -\overline{(u)}_{ S }|^{1+\delta}d \lambda \right)^{\frac{1}{1+\delta}} \\
	& \leq  \liminf \limits_{m \rightarrow \infty} ||T(t_{m})u-\overline{(u)}_{ S }||_{L^{1+\delta}( S )}
	\end{align*}
	which implies (\ref{ubtheoremeq1}) for every $t \in (0,\infty)$ and $u \in L^{2}( S )\cap L^{p}( S )$.\\
	
	Finally, let $t \in (0,\infty)$ be arbitrary and let  $u \in L^{2}( S )\cap L^{1+\delta}( S )$.\\
	Moreover, let $(u_{m})_{m \in \mathbb{N}} \subseteq L^{2}( S )\cap L^{p}( S )$ be such that $\lim \limits_{m \rightarrow \infty}  u_{m}=u$ in $L^{2}( S )$ and in $L^{1+\delta}( S )$. Then it is plain that $\lim \limits_{m \rightarrow \infty} \overline{(u_{m})}_{ S }=\overline{(u)}_{ S }$. Moreover, it follows from Lemma \ref{nicealemma} that $\lim \limits_{m \rightarrow \infty} T(t)u_{m}=T(t)u$  in $L^{1+\delta}( S )$.\\
	Hence
	\begin{align*}
		||T(t)u-\overline{(u)}_{ S }||_{L^{1+\delta}( S )}
		& =  \lim \limits_{m \rightarrow \infty} ||T(t)u_{m}-\overline{(u_{m})}_{ S }||_{L^{1+\delta}( S )} \\
		& \leq \lim \limits_{m \rightarrow \infty}  C_{ S ,1+\delta} \Gamma_{\delta,p} ||u_{m}-\overline{(u_{m})}_{ S }||^{\frac{2}{p}}_{L^{2}( S )} \left(\frac{1}{t}\right)^{\frac{1}{p}}\\
		& =  C_{ S ,1+\delta} \Gamma_{\delta,p}||u-\overline{(u)}_{ S }||^{\frac{2}{p}}_{L^{2}( S )} \left(\frac{1}{t}\right)^{\frac{1}{p}},
	\end{align*}
	which implies (\ref{ubtheoremeq1}) for every $t \in (0,\infty)$ and $u \in L^{2}( S )\cap L^{1+\delta}( S )$. 
\end{proof}

\begin{rem} Whenever $\delta$ is given such that $0 \leq \delta < \frac{p-p_{0}}{p_{0}}$, then $\Gamma_{\delta,p}$ denotes the quantity introduced in \eqref{ubtheoremeq3}.
\end{rem}

\begin{cor}\label{corollarylim} Let $u \in L^{2}( S )$, then
	\begin{align}
	\label{ubtheoremeq2}
	||T(t)u-\overline{(u)}_{ S }||_{L^{1}( S )} \leq C_{ S ,1} \Gamma_{0,p} ||u-\overline{(u)}_{ S }||^{\frac{2}{p}}_{L^{2}( S )} \left(\frac{1}{t}\right)^{\frac{1}{p}}
	\end{align}
	for every $t \in (0,\infty)$.
\end{cor}

The proofs of Lemma \ref{techniacallemma2} and Theorem \ref{mainwithoutasumptions} reveal that one could have stated (with slightly less effort) that
\begin{align}
\label{upperboundrema}
||T(t)u-\overline{(u)}_{ S }||_{L^{1+\delta}( S )} \leq C_{ S ,1+\delta} \Gamma_{\delta,p}||u||^{\frac{2}{p}}_{L^{2}( S )} \left(\frac{1}{t}\right)^{\frac{1}{p}}
\end{align}
for all $u \in L^{2}( S ) \cap L^{1+\delta}( S )$, $0 \leq \delta < \frac{p-p_{0}}{p_{0}}$  and $t \in (0,\infty)$. Note that (\ref{ubtheoremeq1}) is a sharper bound than (\ref{upperboundrema}) since it is well known that
\begin{align*}
||u-\overline{(u)}_{ S }||_{L^{2}( S )} \leq ||u-c||_{L^{2}( S )},~\forall c \in \mathbb{R}.
\end{align*}

If $0 \leq \delta < \frac{p-p_{0}}{p_{0}}$, then $\delta$ can be chosen as bigger as smaller $p_{0}$ gets, i.e. Theorem \ref{mainwithoutasumptions} yields the most general result if $p_{0}=1$. A sufficient condition for this to hold is that there is an $\varepsilon >0$ such that
\begin{align*}
\gamma \geq \varepsilon \text{ a.e. on }  S 
\end{align*}
Particularly, if $\gamma$ is constantly nonzero almost everywhere, then $p_{0}=1$.\\

By virtue of the Sobolev embedding theorem one obtains the main result of this section. 

\begin{thm}\label{mainuniform} Let $u \in L^{p}( S )$ and assume $p_{0}< \frac{p}{n}$, then $T(t)u \in L^{\infty}( S )$ for every $t \in (0,\infty)$. Moreover, if $n-1<\delta<\frac{p-p_{0}}{p_{0}}$, then  $T(t)u \in W^{1,1+\delta}( S )$ and there is a constant $C^{\ast}_{ S ,\delta} \in [0,\infty)$, depending only on $ S $ and $\delta$, such that
	\begin{align}
	\label{mainuniformeq}
	||T(t)u-\overline{(u)}_{ S }||_{L^{\infty}( S )} \leq C^{\ast}_{ S ,\delta} \Gamma_{\delta,p}||u-\overline{(u)}_{ S }||^{\frac{2}{p}}_{L^{2}( S )} \left(\frac{1}{t}\right)^{\frac{1}{p}},
	\end{align}
	for every $t \in (0,\infty)$.\\
	In addition,  $C^{\ast}_{ S ,\delta}$ can be chosen as $C^{\ast}_{ S ,\delta}=\tilde{C}_{ S ,1+\delta}\left( C_{ S ,1+\delta}^{1+\delta}+1\right)^{\frac{1}{1+\delta}}$, where $\tilde{C}_{ S ,1+\delta}$ is the operator norm of the continuous injection $W^{1,1+\delta}( S )\hookrightarrow L^{\infty}( S )$.
\end{thm}
\begin{proof} First of all note that if $p_{0}< \frac{p}{n}$, then $\frac{p-p_{0}}{p_{0}}>n-1$, consequently $(n-1,\frac{p-p_{0}}{p_{0}}) \neq \emptyset$.\\
	So let $n-1<\delta<\frac{p-p_{0}}{p_{0}}$ and $u \in L^{p}( S )$ which implies $u \in L^{2}( S )$, since $p>np_{0} \geq n \geq 2$.\\
	Now observe  that Lemma \ref{techniacallemma2} implies  $T(t)u \in L^{2}( S ) \cap W^{1,p}_{\gamma}( S )$ for a.e.\linebreak $t \in (0,\infty)$ and consequently Lemma \ref{technicallemma1} yields  $T(t)u \in W^{1,1+\delta}( S )$ and moreover
	\begin{align}
	\label{mmainuniform_1}
	||\nabla T(t)u||_{ L^{1+\delta}( S ;\mathbb{R}^{n})} \leq \Gamma_{\delta,p} ||u-\overline{(u)}_{ S }||^{\frac{2}{p}}_{L^{2}( S )} \left(\frac{1}{t}\right)^{\frac{1}{p}} 
	\end{align}
	for a.e. $t \in (0,\infty)$.\\
	Since $T(t)u \in W^{1,1+\delta}( S )$ it is clear that $T(t)u-\overline{(u)}_{ S } \in W^{1,1+\delta}( S )$ , consequently, since $1+\delta>n$, the Sobolev embedding theorem yields
	\begin{align}
	\label{mmainuniform_4}
	||T(t)u-\overline{(u)}_{ S }||_{L^{\infty}( S )} \leq \tilde{C}_{ S ,1+\delta} ||T(t)u-\overline{(u)}_{ S }||_{W^{1,1+\delta}( S )}
	\end{align}
	for almost every $t \in (0,\infty)$, where $\tilde{C}_{ S ,1+\delta}$ is the operator norm of the continuous injection $W^{1,1+\delta}( S )\hookrightarrow L^{\infty}( S )$. \\
	Hence it follows by virtue of Theorem \ref{mainwithoutasumptions},  and the inequalities \eqref{mmainuniform_1} and \eqref{mmainuniform_4} that
	\begin{align*}
		& 
		\left(\frac{1}{\tilde{C}_{ S ,1+\delta}}||T(t)u-\overline{(u)}_{ S }||_{L^{\infty}( S )} \right)^{1+\delta}\\
		&\leq ||T(t)u-\overline{(u)}_{ S }||^{1+\delta}_{W^{1,1+\delta}( S )}   \\
		&= ||T(t)u-\overline{(u)}_{ S }||_{L^{1+\delta}( S )}^{1+\delta}+||\nabla (T(t)u-\overline{(u)}_{ S })||_{L^{1+\delta}( S ;\mathbb{R}^{n})}^{1+\delta} \\
		&= ||T(t)u-\overline{(u)}_{ S }||_{L^{1+\delta}( S )}^{1+\delta}+||\nabla T(t)u||_{L^{1+\delta}( S ;\mathbb{R}^{n})}^{1+\delta} \\
		&\leq \left( C_{ S ,1+\delta}^{1+\delta}+1\right) \left(\Gamma_{\delta,p} ||u-\overline{(u)}_{ S }||^{\frac{2}{p}}_{L^{2}( S )} \left(\frac{1}{t}\right)^{\frac{1}{p}} \right)^{1+\delta} \\
	\end{align*}
	Consequently, if one defines $C^{\ast}_{ S ,\delta}:=\tilde{C}_{ S ,1+\delta}\left( C_{ S ,1+\delta}^{1+\delta}+1\right)^{\frac{1}{1+\delta}}$, then the preceding estimate yields the claim for almost every $t \in (0,\infty)$.\\
	Now let $t \in (0,\infty)$ and choose a monotonically increasing sequence \linebreak$(t_{m})_{m \in \mathbb{N}} \subseteq (0,\infty)$ such that $\lim \limits_{m \rightarrow \infty}t_{m}=t$, $t_{m}<t$ and such that (\ref{mainuniformeq}) holds  for each $m \in \mathbb{N}$. Then Lemma \ref{nicealemma}, together with Lemma \ref{additionlemma}, yield
	\begin{align*}
	||T(t)u-\overline{(u)}_{ S }||_{L^{\infty}( S )} \leq ||T(t_{m})(u-\overline{(u)}_{ S })||_{L^{\infty}( S )} =   ||T(t_{m})u-\overline{(u)}_{ S }||_{L^{\infty}( S )},
	\end{align*}
	for every $m \in \mathbb{N}$, which verifies the claim for every $t \in (0,\infty)$.
\end{proof}

\begin{rem}  Assume $u \in L^{p}( S )$ and $p_{0}< \frac{p}{n}$.  Moreover let $n-1<\delta<\frac{p-p_{0}}{p_{0}}$. Then the preceding theorem states particularly that $T(t)u \in W^{1,1+\delta}( S )$. Consequently, the Sobolev embedding theorem also yields that $T(t)u$ is H\"older continuous of order $1-\frac{n}{1+\delta}$, or more accurately that there is a representative in the equivalence class which is H\"older continuous of this order. 
\end{rem}

\begin{rem} It is clear that Corollary \ref{corollarylim} implies
	\begin{align*}
	\lim \limits_{t \rightarrow \infty} ||T(t)u-\overline{(u)}_{ S }||_{L^{1}( S )}=0,~\forall u \in L^{2}( S ).
	\end{align*}
	Moreover, Theorem \ref{mainuniform} yields that this convergence is even uniform, if \linebreak$u \in L^{2}( S )\cap L^{p}( S )$ and $p_{0}< \frac{p}{n}$.\\ 
	It is beyond the scope of this paper to obtain a uniform convergence result under more general assumptions. But it will be proven that $L^{q}$-convergence holds under more general assumptions for any $q \in [1,\infty)$. 
\end{rem}

\begin{thm}\label{almostfinallqlemma} Let $q \in [1,\infty)$ and $u \in L^{q}( S )$, then
	\begin{align}
	\lim \limits_{t \rightarrow \infty} T(t)u=\overline{(u)}_{ S } \text{ in } L^{q}( S ).
	\end{align}
\end{thm}
\begin{proof} Let $q \in [1,\infty)$, $u \in L^{q}( S )$, $k \in (0,\infty)$ and let $\tau_{k}:\mathbb{R} \rightarrow \mathbb{R}$ denote the standard truncation function.\\
	Let $(\tilde{t}_{m})_{m \in \mathbb{N}} \subseteq [0,\infty)$ be an arbitrary sequence such that $\lim \limits_{m \rightarrow \infty}\tilde{t}_{m}=\infty$. Moreover, let $(t_{m})_{m \in \mathbb{N}}$ be a subsequence such that
	\begin{align}
	\label{almostfinallqlemmaproof1}
	\lim \limits_{m \rightarrow \infty} T(t_{m})\tau_{k}(u)= \overline{(\tau_{k}(u))}_{ S },\text{ a.e. on }  S 
	\end{align}
	(Corollary \ref{corollarylim} ensures the existence of such a subsequence, since \linebreak$\tau_{k}(u) \in L^{2}( S )$.)\\
	Now observe that Lemma \ref{nicealemma} implies 
	\begin{align*}
	||T(t_{m})\tau_{k}(u)-\overline{(\tau_{k}(u))}_{ S }||_{L^{\infty}( S )} \leq 2k, 
	\end{align*}
	for all $m \in \mathbb{N}$. Consequently, this, together with (\ref{almostfinallqlemmaproof1}) yields, by virtue of dominated convergence, that $\lim \limits_{m \rightarrow \infty} T(t_{m})\tau_{k}(u)=\overline{(\tau_{k}(u))}_{ S }$ in $L^{q}( S )$ and therefore
	\begin{align}
	\label{almostfinallqlemmaproof3}
	\lim \limits_{t \rightarrow \infty} ||T(t)\tau_{k}(u)-\overline{(\tau_{k}(u))}_{ S }||_{L^{q}( S )}=0,~\forall k \in (0,\infty).
	\end{align} 
	Observe that clearly $\lim \limits_{k \rightarrow \infty} \tau_{k}(u)=u$ a.e. on $ S $ and that $|\tau_{k}(u)-u|^{q} \leq (2|u|)^{q}$ for all $k  \in (0,\infty)$. Consequently Lebesgue's theorem yields
	\begin{align}
	\label{almostfinallqlemmaproof2}
	\lim \limits_{k \rightarrow \infty} \tau_{k}(u)=u,~\text{in   } L^{q}( S ).
	\end{align}
	Now let $\varepsilon>0$ and choose $k_{0} \in (0,\infty)$ sufficiently large such that
	\begin{align}
	\label{almostfinallqlemmaproof5}
	\max(||\tau_{k_{0}}(u)-u||_{L^{q}( S )},||\overline{(\tau_{k_{0}}(u))}_{ S }-\overline{(u)}_{ S }||_{L^{q}( S )})<\frac{\varepsilon}{3},
	\end{align}
	which is possible, due to (\ref{almostfinallqlemmaproof2}).\\
	Moreover, (\ref{almostfinallqlemmaproof3}) yields the existence of $t_{0} \in (0,\infty)$ such that
	\begin{align}
	\label{almostfinallqlemmaproof4}
	||T(t)\tau_{k_{0}}(u)-\overline{(\tau_{k_{0}}(u))}_{ S }||_{L^{q}( S )} < \frac{\varepsilon}{3},~\forall t \geq t_{0}.
	\end{align}
	Finally, it follows by combining (\ref{almostfinallqlemmaproof5}), (\ref{almostfinallqlemmaproof4}) and by using Lemma \ref{nicealemma}, that
	\begin{align*}
		||T(t)u-\overline{(u)}_{ S }||_{L^{q}( S )}<  \varepsilon
	\end{align*}
	for all $t \geq t_{0}$.
\end{proof}

\section{Extinction of solutions}

The basic idea to prove extinction of solutions is to apply the following lemma, which is stated, but not proven, in \cite{extinction}, Lemma 2.2. Even though this lemma seems to be in common use, the present author was unable to find a proof in the literature. Therefore, the proof will be given.

\begin{lem}\label{importantlemma} Let $k \in (0,1)$, $\alpha \in (0,\infty)$ and let $f:[0,\infty)\rightarrow [0,\infty)$ be locally Lipschitz continuous, i.e.  $f|_{[0,\tilde{t}]}$ is Lipschitz continuous for any $\tilde{t} \in (0,\infty)$. Moreover, assume
	\begin{align*}
	f^{\prime}(t)+\alpha f(t)^{k} \leq 0
	\end{align*}
	for a.e. $t \in (0,\infty)$ and introduce
	\begin{align*}
	T^{\ast}:= \frac{f(0)^{1-k}}{\alpha(1-k)},
	\end{align*}
	then $f(t)=0$ for all $t \in [T^{\ast},\infty)$.
\end{lem}
\begin{proof} Let $k \in (0,1)$, $\alpha \in (0,\infty)$ and let $f,~\tilde{f}:[0,\infty)\rightarrow [0,\infty)$ be locally Lipschitz continuous. Moreover, assume $f(0)=\tilde{f}(0)=:a$ and
	\begin{enumerate}
		\item $f^{\prime}(t)=-\alpha f(t)^{k}$ for a.e. $t \in (0,\infty)$, 
		\item $\tilde{f}^{\prime}(t)\leq-\alpha\tilde{f}(t)^{k}$ for a.e. $t \in (0,\infty)$.
	\end{enumerate}
	It will be proven that $0\leq\tilde{f}(t) \leq f(t)$ for all $t \in [0,\infty)$ which obviously implies that it suffices to prove that $f(t)=0$ for all $t\geq T^{\ast}:=\frac{a^{1-k}}{\alpha(1-k)}$.\\
	Assume there is $t_{1}>0$ such that $\tilde{f}(t_{1})>f(t_{1})$, then there is, since both functions are continuous and since $f(0)=\tilde{f}(0)$, a $t_{0} \in [0,t_{1})$ such that
	\begin{align}
	\label{niceextinctlemmaproof1}
	\tilde{f}(t)>f(t),~\forall t \in (t_{0},t_{1}] \text{ and } \tilde{f}(t_{0})=f(t_{0}).
	\end{align}
	But this implies
	\begin{align*}
		f(t_{1})-\tilde{f}(t_{1})
		& =  f(t_{1})-\tilde{f}(t_{1}) -(f(t_{0})-\tilde{f}(t_{0})) \\
		& =  \int \limits_{t_{0}} \limits^{t_{1}}f^{\prime}(t)-\tilde{f}^{\prime}(t)dt\\ 
		& \geq  \int \limits_{t_{0}} \limits^{t_{1}}-\alpha f(t)^{k}+\alpha\tilde{f}(t)^{k}dt\\ 
		& \geq  0,
	\end{align*}
	which yields $f(t_{1})\geq\tilde{f}(t_{1})$ and therefore contradicts (\ref{niceextinctlemmaproof1}).\\
	Now it will be proven that $f(t)=0$ for all $t\geq T^{\ast}:=\frac{a^{1-k}}{\alpha(1-k)}$ which then implies the claim. \\
	First of all note that $f^{\prime}$ can be extended to a continuous function on $[0,\infty)$. Consequently, $f$ is continuously differentiable on $(0,\infty)$.\\ 
	Moreover, $f^{\prime} \leq 0$ which yields that $f$ is monotonically decreasing. Hence, if $f(\tau)=0$ then $f(t)=0$ for all $t \geq \tau$, since $f \geq 0$ by assumption.\\
	Now introduce $\tau:=\inf \{ t \geq 0 :~f(t)=0\}$. The claim follows if $\tau \leq T^{\ast}$.\\ 
	Consequently let us contradict $\tau > T^{\ast}$.  If $\tau > T^{\ast}$ then $f(t)>0$ for all $t \in [0,T^{\ast}]$ and consequently $\frac{f^{\prime}(t)}{-\alpha f(t)^{k}}=1$ for all $t \in [0,T^{\ast}]$, which yields by substituting that
	\begin{align*}
		T^{\ast}
		& = \int \limits_{0} \limits^{T^{\ast}} \frac{f^{\prime}(t)}{-\alpha f(t)^{k}} dt \\
		& = \int \limits_{f(0)} \limits^{f(T^{\ast})} \frac{1}{-\alpha t^{k}} dt \\
		& = \frac{f(0)^{1-k}}{\alpha(1-k)}- \frac{1}{\alpha(1-k)} f(T^{\ast})^{1-k} \\
		& = T^{\ast}-  \frac{1}{\alpha} \frac{1}{1-k} f(T^{\ast})^{1-k} \\
	\end{align*}
	and consequently $f(T^{\ast})=0$ which contradicts  $\tau > T^{\ast}$.
\end{proof}

Here and in everything which follows let $f_{u}:[0,\infty)\rightarrow [0,\infty)$ be defined by
\begin{align*}
f_{u}(t):= \int \limits_{ S } \left(T(t)u-\overline{(u)}_{ S }\right)^{2}d \lambda
\end{align*}
for any $t \in [0,\infty)$ and $u \in L^{2}( S )$. 

\begin{lem} Let $u \in D(A)$, then $f_{u}$ is locally Lipschitz continuous. 
\end{lem}
\begin{proof} Let $u \in D(A)$ and $\tilde{t}>0$ be given. Moreover, let $L$ denote the Lipschitz constant of $[0,\tilde{t}] \ni t \mapsto T(t)u \in (L^{1}( S ),||\cdot||_{L^{1}( S )})$. (As $u \in D(A)$, the Lipschitz continuity follows from \cite{BenilanBook}, Lemma 7.8.)\\
	Now Lemma \ref{nicealemma} yields that
	\begingroup
	\allowdisplaybreaks
	\begin{align*}
		&  
		|f_{u}(t_{1})-f_{u}(t_{2})|\\ 
		& =  \left| \int \limits_{ S } \left(T(t_{1})u-\overline{(u)}_{ S }\right)^{2}- \left(T(t_{2})u-\overline{(u)}_{ S }\right)^{2}d \lambda \right|\\
		& \leq   \int \limits_{ S }\left| \left(T(t_{1})u\right)^{2}-\left(T(t_{2})u\right)^{2}-2\overline{(u)}_{ S }T(t_{1})u+2\overline{(u)}_{ S }T(t_{2})u\right|d \lambda \\
		& \leq  \left|\left| \left(T(t_{1})u+T(t_{2})u\right)\left(T(t_{1})u-T(t_{2})u\right)\right|\right|_{L^{1}( S )}+2|\overline{(u)}_{ S }|L|t_{1}-t_{2}|\\
		& \leq  \left|\left|T(t_{1})u+T(t_{2})u\right|\right|_{L^{\infty}( S )}\left|\left|T(t_{1})u-T(t_{2})u\right|\right|_{L^{1}( S )}+2|\overline{(u)}_{ S }|L|t_{1}-t_{2}|\\
		& \leq  2||u||_{L^{\infty}( S )}L|t_{1}-t_{2}|+2|\overline{(u)}_{ S }|L|t_{1}-t_{2}|\\
		& =  2L(||u||_{L^{\infty}( S )}+|\overline{(u)}_{ S }|)|t_{1}-t_{2}|\\
	\end{align*}
	\endgroup
	for any $t_{1},t_{2} \in [0,\tilde{t}]$. (Note that indeed $u \in L^{\infty}( S )$, since $u \in D(A)$.)
\end{proof}

\begin{lem}\label{difflemma} Let $u \in D(A)$, then $f_{u}$ is differentiable almost everywhere on $(0,\infty)$ and
	\begin{align}
	\label{difflemmaeq}
	f_{u}^{\prime}(t)= -2||\nabla T(t)u||_{L^{p}( S ,\nu;\mathbb{R}^{n})}^{p}
	\end{align}
	for almost every $t \in (0,\infty)$.
\end{lem}
\begin{proof} Let $u \in D(A)$, $v:=u-\overline{(u)}_{ S }$. Let $t \in (0,\infty)$ be such that $T(\cdot)v$ is differentiable at $t$ and let $(h_{m})_{m \in \mathbb{N}} \subseteq (0,\infty)$ such that $\lim \limits_{m \rightarrow \infty}h_{m} = 0$.\\
	It is clear that
	\begin{align*}
	\lim \limits_{m \rightarrow \infty} \frac{T(t+h_{m})v-T(t)v}{h_{m}} = T^{\prime}(t)v \text{ and } \lim \limits_{m \rightarrow \infty} T(t+h_{m})v+T(t)v=2T(t)v .
	\end{align*}	
	in $L^{1}( S )$.\\
	Consequently, by passing to a subsequence if necessary, this convergences holds also almost everywhere, which yields
	\begin{align}
	\label{diffproof1}
	\lim \limits_{m \rightarrow \infty} \frac{(T(t+h_{m})v)^{2}-(T(t)v)^{2}}{h_{m}} =  2T(t)v T^{\prime}(t)v \text{ a.e. on }  S .
	\end{align}
	It follows from \cite{cao}, Theorem 4.2 and 4.4 that
	\begin{align*}
	\left|\left| \frac{(T(t+h_{m})v)^{2}-(T(t)v)^{2}}{h_{m}} \right|\right|_{L^{\infty}( S )}  \leq \frac{4||v||^{2}_{L^{\infty}( S )}}{|p-2|t},
	\end{align*}
	for all $m \in \mathbb{N}$. This, together with (\ref{diffproof1}) implies, by virtue of dominated convergence, that
	\begin{align}
	\lim \limits_{m \rightarrow \infty} \frac{f_{u}(t+h_{m})-f_{u}(t)}{h_{m}} =2\int \limits_{ S }T(t)\left(u-\overline{(u)}_{ S }\right) T^{\prime}(t)\left(u-\overline{(u)}_{ S }\right)d\lambda 
	\end{align}
	and consequently one infers, by using Lemma \ref{additionlemma} and Lemma \ref{nicealemma}, that
	\begin{align*}
		\lim \limits_{m \rightarrow \infty} \frac{f_{u}(t+h_{m})-f_{u}(t)}{h_{m}} 
		& =  2\int \limits_{ S }\left(T(t)(u)-\overline{(u)}_{ S }\right) T^{\prime}(t)u d\lambda\\ 
		& = 2\int \limits_{ S }T(t)(u) T^{\prime}(t)ud\lambda - 2\int \limits_{ S }\overline{(u)}_{ S } T^{\prime}(t)ud\lambda\\
		& =  -2||\nabla T(t)u||_{L^{p}( S ,\nu;\mathbb{R}^{n})}^{p}.
	\end{align*}
	The preceding calculation yields that the right derivative of $f_{u}$ is given by the right hand side of (\ref{difflemmaeq}). Consequently (\ref{difflemmaeq}) holds, since any real valued, locally Lipschitz continuous function is differentiable almost everywhere. 
\end{proof}

\begin{lem}\label{extinctionlemma} Let $u \in D(A)$ and assume that the interval $\left(\frac{p_{0}(n-2)}{n+2}+p_{0},2\right)$ is nonempty. Moreover, assume $p \in \left(\frac{p_{0}(n-2)}{n+2}+p_{0},2\right)$, then there is a constant $T^{\ast}_{u,\gamma,p,n, S }$ such that
	\begin{align*} 
	T(t)u=\overline{(u)}_{ S } \text{ a.e. on }  S 
	\end{align*}
	for all $t \geq T^{\ast}_{u,\gamma,p,n, S }$.\\
	In addition, $T^{\ast}_{u,\gamma,p,n, S }$ can be chosen as
	\begin{align*}
	T^{\ast}_{u,\gamma,p,n, S } := \frac{\left( \int \limits_{ S } (u-\overline{(u)}_{ S })^{2} d \lambda \right)^{1-\frac{p}{2}}}{2-p}  \tilde{C}_{ S }^{p}\left(C_{ S ,\frac{2n}{n+2}}^{\frac{2n}{n+2}}+1\right)^{\frac{np+2p}{2n}}\tilde{\Gamma}_{n,p} <\infty,
	\end{align*}
	where the constant $\tilde{C}_{ S }$ denotes the operator norm of the continuous injection \linebreak$W^{1,\frac{2n}{n+2}}( S ) \hookrightarrow L^{2}( S )$ and
	\begin{align*}
		\tilde{\Gamma}_{n,p}:=\left( \int \limits_{ S } \gamma^{\frac{2n}{2n-np-2p}}d\lambda \right)^{\frac{np+2p-2n}{2n}}<\infty.
	\end{align*}
\end{lem}
\begin{proof} Let $u \in D(A)$, $p \in \left(\frac{p_{0}(n-2)}{n+2}+p_{0},2\right)$ and assume that this interval is nonempty.\\
	First of all note that $\frac{2n}{n+2}<n$, since $n \neq 1$. Consequently, Sobolev's embedding theorem  yields that there is a continuous injection $W^{1,\frac{2n}{n+2}}( S ) \hookrightarrow L^{2}( S )$. So let $\tilde{C}_{ S }$ denote its operator norm.\\
	Now let $t \in (0,\infty)$ be such that $-T^{\prime}(t)u=AT(t)u$ and such that (\ref{difflemmaeq}) holds. (Clearly a.e. point in $(0,\infty)$ is a valid choice for $t$.)\\
	Note that
	\begin{align}
	\label{fancyproof1}
	0 \leq \frac{2n}{n+2}-1= \frac{1}{p_{0}}  \left( \frac{p_{0}(n-2)}{n+2}+p_{0}\right)-1<\frac{p}{p_{0}}-1=\frac{p-p_{0}}{p_{0}}
	\end{align}
	Moreover, $T(t)u \in D(A)$ yields $T(t)u \in W^{1,p}_{\gamma}( S )$ and consequently it follows by virtue of Lemma \ref{technicallemma1} and (\ref{fancyproof1}), that $T(t)u \in W^{1,\frac{2n}{n+2}}( S )$ and
	\begin{align}
	\label{fancyproof2}
	||\nabla T(t)u||^{p}_{L^{\frac{2n}{n+2}}( S ;\mathbb{R}^{n})} \leq \tilde{\Gamma}_{n,p} ||\nabla T(t)u||_{L^{p}( S ,\nu;\mathbb{R}^{n})}^{p}
	\end{align}
	and particularly that $\int \limits_{ S } \gamma^{\frac{2n}{2n-np-2p}} d\lambda < \infty$ which implies that $T^{\ast}_{u,\gamma,p,n, S }<\infty$.\\
	Now introduce 
	\begin{align*}
	\alpha_{\gamma,p,n, S }:= 2 \left(\tilde{C}_{ S }^{p} \left( C_{ S ,\frac{2n}{n+2}}^{\frac{2n}{n+2}}+1\right)^{\frac{np+2p}{2n}} \tilde{\Gamma}_{n,p} \right)^{-1},
	\end{align*}
	then
	\begin{align}
	\label{fancyproof4}
	f_{u}(t)^{\frac{p}{2}} \leq 2 \alpha_{\gamma,p,n, S }^{-1}||\nabla T(t)u||^{p}_{L^{p}( S ,\nu;\mathbb{R}^{n})},
	\end{align}
	since 
	\begin{align*}
		f_{u}(t)^{\frac{p}{2}}
		& =  ||T(t)u-\overline{(u)}_{ S }||_{L^{2}( S )}^{p} \\
		& \leq  \tilde{C}_{ S }^{p} ||T(t)u-\overline{(u)}_{ S }||_{W^{1,\frac{2n}{n+2}}( S )}^{p} \\
		& =  \tilde{C}_{ S }^{p} \left(||T(t)u-\overline{(u)}_{ S }||_{L^{\frac{2n}{n+2}}( S )}^{\frac{2n}{n+2}}+||\nabla T(t)u||_{L^{\frac{2n}{n+2}}( S ;\mathbb{R}^{n})}^{\frac{2n}{n+2}} \right)^{\frac{np+2p}{2n}} \\
		& \leq  \tilde{C}_{ S }^{p} \left( C_{ S ,\frac{2n}{n+2}}^{\frac{2n}{n+2}}||\nabla T(t)u||_{L^{\frac{2n}{n+2}}( S ;\mathbb{R}^{n})}^{\frac{2n}{n+2}}+||\nabla T(t)u||_{L^{\frac{2n}{n+2}}( S ;\mathbb{R}^{n})}^{\frac{2n}{n+2}} \right)^{\frac{np+2p}{2n}} \\
		& =  \tilde{C}_{ S }^{p} \left( C_{ S ,\frac{2n}{n+2}}^{\frac{2n}{n+2}}+1\right)^{\frac{np+2p}{2n}}||\nabla T(t)u||_{L^{\frac{2n}{n+2}}( S ;\mathbb{R}^{n})}^{p} \\ 
		& \leq  \tilde{C}_{ S }^{p} \left( C_{ S ,\frac{2n}{n+2}}^{\frac{2n}{n+2}}+1\right)^{\frac{np+2p}{2n}} \tilde{\Gamma}_{n,p} ||\nabla T(t)u||^{p}_{L^{p}( S ,\nu;\mathbb{R}^{n})} \\
		& =  2 \alpha_{\gamma,p,n, S }^{-1}||\nabla T(t)u||^{p}_{L^{p}( S ,\nu;\mathbb{R}^{n})},
	\end{align*}
	where the Sobolev embedding theorem, Poincar\'{e}'s inequality and (\ref{fancyproof2}) have been used.\\
	Consequently, (\ref{fancyproof4}) and Lemma \ref{difflemma} yield
	\begin{align*} 
	f^{\prime}_{u}(t)+\alpha_{\gamma,p,n, S } f_{u}(t)^{\frac{p}{2}} \leq -2||\nabla T(t)u||_{L^{p}( S ,\nu;\mathbb{R}^{n})}^{p}+2||\nabla T(t)u||_{L^{p}( S ,\nu;\mathbb{R}^{n})}^{p} = 0
	\end{align*}
	Conclusively, Lemma \ref{importantlemma} yields that $f_{u}(t)=0$ for all 
	\begin{align*}
	t \geq \frac{f_{u}(0)^{1-\frac{p}{2}}}{\alpha_{\gamma,p,n, S }(1-\frac{p}{2})}=T^{\ast}_{u,\gamma,p,n, S },
	\end{align*}
	which implies the claim, since $f_{u}(t)=0$ for all $t \geq T^{\ast}_{u,\gamma,p,n, S }$ clearly yields $T(t)u=\overline{(u)}_{ S }$ a.e. on $ S $ for all $t \geq T^{\ast}_{u,\gamma,p,n, S }$. 
\end{proof}

\begin{rem} Whenever $p \in \left(\frac{p_{0}(n-2)}{n+2}+p_{0},2\right) \neq \emptyset$ and $u \in L^{2}( S )$ then $T^{\ast}_{u,\gamma,p,n, S }$ and $\tilde{\Gamma}_{n,p}$ denote the constants defined in Lemma \ref{extinctionlemma}.\\
	The proof of the preceding lemma reveals that these are indeed finite.
\end{rem}

So far one only knows that $D(A)$ is a dense subset of $(L^{1}( S ),||\cdot||_{L^{1}( S )})$. This result is of course not very useful to generalize the preceding  Lemma to more general initial values than $u \in D(A)$. It will be established now that $D(A)$ is even a dense subset of $(L^{2}( S ),||\cdot||_{L^{2}( S )})$. The applied technique is the same as in \cite{mainbasedon}, Prop. 5.1. 

\begin{lem} $D(A)$ is a dense subset of $(L^{2}( S ),||\cdot||_{L^{2}( S )})$.
\end{lem}
\begin{proof} It suffices to prove that there is for each $h \in L^{\infty}( S )$ a sequence $(f_{m})_{m \in \mathbb{N}} \subseteq D(A)$ such that
	\begin{align*}
	\lim \limits_{m \rightarrow \infty}f_{m}=h~\text{in } L^{2}( S ),
	\end{align*}
	since $L^{\infty}( S )$ is a dense subspace of $L^{2}( S )$.\\
	Let $h \in L^{\infty}( S )$ be arbitrary but fixed.\\
	Since $\mathcal{A}$  is  m-accretive there are for each $m \in \mathbb{N}$ functions $f_{m}\in D(\mathcal{A})$, $\hat{f}_{m}\in \mathcal{A}f_{m}$, such that
	\begin{align}
	\label{basicallyfancyproof2}
	h=f_{m}+\frac{1}{m}\hat{f}_{m}\text{ a.e. on }  S 
	\end{align}
	for all $m \in \mathbb{N}$.\\
	By complete accretivity one obtains $f_{m}<<f_{m}+\frac{1}{m}\hat{f}_{m}$ and consequently $f_{m}<<h$ for all $m \in \mathbb{N}$, which yields
	\begin{align}
	\label{basicallyfancyproof4}
	||f_{m}||_{L^{\infty}( S )} \leq ||h||_{L^{\infty}( S )}<\infty,~\forall m \in \mathbb{N}.
	\end{align}
	Consequently $f_{m} \in L^{\infty}( S )$ and therefore $f_{m} \in D(A)$ for all $m \in \mathbb{N}$.\\ 
	Moreover, (\ref{basicallyfancyproof4}) also implies that the sequence $(||f_{m}||_{L^{2}( S )})_{m \in \mathbb{N}}$ is bounded. Hence, by passing to a subsequence if necessary, there is an $\tilde{h}\in L^{2}( S )$ such that
	\begin{align}
	\label{basicallyfancyproof8}
	\wlim \limits_{m \rightarrow \infty} f_{m}=\tilde{h}~\text{in } L^{2}( S ).
	\end{align}
	Now observe that
	\begin{align}
	\label{basicallyfancyproof1}
	\lim \limits_{m \rightarrow \infty} \frac{1}{m}\int \limits_{ S } \gamma |\nabla f_{m}|^{p-2} \nabla f_{m}\cdot\nabla \varphi d \lambda=0,~\forall \varphi \in W^{1,p}_{\gamma}( S )\cap L^{\infty}( S ),
	\end{align}
	since one obtains for all $\varphi \in W^{1,p}_{\gamma}( S )\cap L^{\infty}( S )$ and $q:=\frac{p}{p-1}$ that
	\begin{align*}
		&  
		\left|\left(\frac{1}{m}\right)^{\frac{1}{q}}\int \limits_{ S } \gamma |\nabla f_{m}|^{p-2} \nabla f_{m}\cdot\nabla \varphi d \lambda\right|\\ 
		& \leq  \left(\frac{1}{m}\right)^{\frac{1}{q}}\left(\int \limits_{ S } \gamma |\nabla f_{m}|^{p-2} \nabla f_{m}\cdot\nabla f_{m}d\lambda \right)^{\frac{1}{q}} ||\nabla \varphi||_{L^{p}( S ,\nu;\mathbb{R}^{n})}\\
		& =  \left(\int \limits_{ S } (h-f_{m})f_{m} d \lambda \right)^{\frac{1}{q}}||\nabla \varphi||_{L^{p}( S ,\nu;\mathbb{R}^{n})}\\
		& \leq \left(\int \limits_{ S } (||h||_{L^{\infty}( S )}+||h||_{L^{\infty}( S )})||h||_{L^{\infty}( S )} d \lambda \right)^{\frac{1}{q}} ||\nabla \varphi||_{L^{p}( S ,\nu;\mathbb{R}^{n})}\\
		& =  \left(2\lambda( S )||h||^{2}_{L^{\infty}( S )} \right)^{\frac{1}{q}} ||\nabla \varphi||_{L^{p}( S ,\nu;\mathbb{R}^{n})},
	\end{align*}
	where Cauchy Schwarz inequality, H\"older's inequality, $\hat{f}_{m}=Af_{m}$, (\ref{basicallyfancyproof2}) and (\ref{basicallyfancyproof4}) were used.\\
	Moreover, (\ref{basicallyfancyproof1}) yields 
	\begin{align*}
		\int \limits_{ S } (h-\tilde{h})\varphi d \lambda
		& =  \lim \limits_{m \rightarrow \infty}\int \limits_{ S } (h-f_{m})\varphi d \lambda\\
		& =  \lim \limits_{m \rightarrow \infty}\int \limits_{ S } \frac{1}{m}\hat{f}_{m}\varphi d \lambda\\  
		& =  \lim \limits_{m \rightarrow \infty}\frac{1}{m} \int \limits_{ S } \gamma |\nabla f_{m}|^{p-2} \nabla f_{m}\cdot\nabla \varphi d \lambda\\
		& =  0.
	\end{align*}
	for all $\varphi \in  W^{1,p}_{\gamma}( S ) \cap L^{\infty}( S )$ and therefore $h=\tilde{h}$.\\
	It is clear that $||f_{m}||_{L^{2}( S )}\leq||f_{m}+\frac{1}{m}\hat{f}_{m}||_{L^{2}( S )}$ and consequently one gets $||f_{m}||_{L^{2}( S )} \leq ||h||_{L^{2}( S )}=||\tilde{h}||_{L^{2}( S )}$ for all $m \in \mathbb{N}$, which implies particularly that
	\begin{align*}
	\limsup_{m \rightarrow \infty} ||f_{m}||_{L^{2}( S )} \leq ||\tilde{h}||_{L^{2}( S )}. 
	\end{align*}
	Conclusively this, (\ref{basicallyfancyproof8}) and the uniform convexity of the Banach space $L^{2}( S )$ yield $\lim \limits_{m \rightarrow \infty} f_{m}=\tilde{h}=h,~\text{in } L^{2}( S )$.
\end{proof}

\begin{thm} Let $u \in L^{2}( S )$ and assume that the interval $\left(\frac{p_{0}(n-2)}{n+2}+p_{0},2\right)$ is nonempty. Moreover, assume $p \in \left(\frac{p_{0}(n-2)}{n+2}+p_{0},2\right)$, then    
	\begin{align*}
	T(t)u=\overline{(u)}_{ S } \text{ a.e. on }  S 
	\end{align*}
	for all $t \geq T^{\ast}_{u,\gamma,p,n, S }$. 
\end{thm}
\begin{proof} Let $u \in L^{2}( S )$ be given and assume that $u$ is not constant a.e. on $ S $. (If $u$ is constant the claim is trivial.)\\ 
	Now let $(v_{m})_{m \in \mathbb{N}} \subseteq D(A)$ be such that $\lim \limits_{m \rightarrow \infty}v_{m}=u$ in $L^{2}( S )$ and assume that none of the $v_{m}$ is constant a.e. on $ S $.\\
	Moreover, introduce $(u_{m})_{m \in \mathbb{N}}$ by
	\begin{align}
	u_{m}:= \frac{||u-\overline{(u)}_{ S }||_{L^{2}( S )}}{||v_{m}-\overline{(v_{m})}_{ S }||_{L^{2}( S )}}v_{m},~\forall m \in \mathbb{N}.
	\end{align}
	It is clear that $\lim \limits_{m \rightarrow \infty}u_{m}=u$ in $L^{2}( S )$ and that $T^{\ast}_{u_{m},\gamma,p,n, S }=T^{\ast}_{u,\gamma,p,n, S }$ for all $m \in \mathbb{N}$.\\
	Observe that also $u_{m} \in D(A)$ for all $m \in \mathbb{N}$. (Generally if $(f,\hat{f}) \in A$ then $(\alpha f, \alpha^{p-1}\hat{f}) \in A$ for any $\alpha >0$.)\\
	Consequently Lemma \ref{extinctionlemma} yields $T(t)u_{m}=\overline{(u_{m})}_{ S }$ a.e. on $ S $ for every \linebreak$t \geq T^{\ast}_{u,\gamma,p,n, S }$.\\
	Finally observe that$\lim \limits_{m \rightarrow \infty}T(t)u_{m}=T(t)u$ in $L^{2}( S )$ for any $t \in [0,\infty)$, which clearly implies the claim.
\end{proof}

Using the preceding result and Theorem \ref{mainuniform} one obtains the following corollary for the case $n=2$ and $p_{0}=1$ which concludes this paper.\\
Note that this corollary is applicable for any $p \in (1,\infty) \setminus \{2\}$, i.e. for any value of $p$ for which the existence of unique strong solutions of (\ref{mainoldeq}) is proven.

\begin{cor} Assume $n=2$ and $p_{0}=1$, then
	\begin{align*}
	T(t)u=\overline{(u)}_{ S } \text{ a.e. on }  S ,
	\end{align*}
	for all $t \geq T^{\ast}_{u,\gamma,p,2, S }$, if $p \in (1,2)$ and $u \in L^{2}( S )$.\\
	Moreover, if $p \in (2,\infty)$ and $u \in L^{p}( S )$, then
	\begin{align*} 
	||T(t)u-\overline{(u)}_{ S }||_{L^{\infty}( S )} \leq C^{\ast}_{ S ,\delta} \Gamma_{\delta,p} ||u-\overline{(u)}_{ S }||^{\frac{2}{p}}_{L^{2}( S )} \left(\frac{1}{t}\right)^{\frac{1}{p}}
	\end{align*}
	for every $t \in (0,\infty)$ and $\delta \in (1,p-1)$.
\end{cor}

% ------------------------------------------------------------------------

\subsection*{Acknowledgment}
The present author is grateful to Prof. Dr. Wolfgang Arendt as well as to Prof. Dr. Evgeny Spodarev for their advices

% ------------------------------------------------------------------------
\end{document}